\documentclass[12pt,a4paper]{article}

\usepackage[utf8]{inputenc}
\usepackage[T1]{fontenc}
\usepackage{lmodern}
\usepackage[indonesian,english]{babel}
\usepackage{amsmath,amssymb,amsthm}
\usepackage{graphicx}
\usepackage{hyperref}
\usepackage{geometry}
\usepackage{cite}
\usepackage{setspace}
\usepackage{enumerate}
\usepackage{caption}
\usepackage{subcaption}

\usepackage{CJKutf8}

\newtheorem{definition}{Definition}[section]
\newtheorem{theorem}[definition]{Theorem}
\newtheorem{lemma}[definition]{Lemma}

\theoremstyle{remark}

\title{\textbf{Generalized Discrete Orlicz-Morrey Spaces}\\[0.5em]}
\author{
  Hukmashabiyya Ariq Gumilar\thanks{Email: hukmashabiyya@upi.edu}, Al Azhary Masta, Siti Fatimah\\
  \small{Mathematics Study Program, Universitas Pendidikan Indonesia, Indonesia}
}

\date{}

\begin{document}

\maketitle

\begin{abstract}
The Orlicz-Morrey spaces, which were introduced through the research of Nakai in 2006, are a generalization and combination of Orlicz and Morrey spaces. There are two types of Orlicz-Morrey spaces, such as continuous Orlicz-Morrey spaces and discrete Orlicz-Morrey spaces. Some properties that apply to Orlicz-Morrey spaces have been studied correspondingly to discrete Orlicz-Morrey spaces. The objectives of the study are to construct generalized discrete Orlicz-Morrey spaces by substituting a Young function with \emph{s}-Young function. Furthermore, The purpose of this study is to see the validity of the properties of the discrete Orlicz-Morrey spaces to the generality of the discrete Orlicz-Morrey spaces. The method in this research draws on the definitions and properties of the discrete Orlicz-Morrey spaces of the previous study and applies the \emph{s}-Young function to the new Orlicz-Morrey spaces. As a result, this study concludes that generalized discrete Orlicz-Morrey spaces reduce to discrete Orlicz-Morrey spaces when \emph{s} is equal to 1. Furthermore, due to the characteristics of the \emph{s}-Young function, some properties of discrete Orlicz-Morrey spaces are preserved in generalized discrete Orlicz-Morrey spaces.
\end{abstract}

\noindent\textbf{Keywords:} Convex-\emph{s} function, Young-\emph{s} function, Orlicz-Morrey sequence spaces.\\
\noindent \textbf{MSC:} 42B35, 46B45, 46A45.

\section{Introduction}
In 2006, Nakai first introduced the Orlicz-Morrey spaces as a generalization and combination of Orlicz and Morrey spaces. There are two types of Orlicz-Morrey spaces that have been discussed by many researchers, namely continuous Orlicz-Morrey spaces and discrete Orlicz-Morrey spaces. Research on continuous Orlicz-Morrey spaces has been carried out by Nakai (2006), Sawano et al. (2012), Deringoz et al. (2014), and Masta et al. (2017) \cite{ref1, ref2, ref3, ref4, ref5, ref6}. Moreover, the Orlicz-Morrey sequence spaces have been studied by Fatimah et al. (2021) \cite{ref7}.\\

According to the study in \cite{ref7}, the discrete Orlicz-Morrey spaces are constructed using Young functions. Further investigation of Young functions has been conducted in \cite{ref6}, where Young function is defined as a function $\Phi : \left[0, \infty\right) \rightarrow \left[0, \infty\right)$ that satisfies the following conditions, such as $\Phi$ is a convex function. That satisfies $\Phi\left(tx+\left(1-t\right)y\right) \leq t\Phi\left(x\right)+\left(1-t\right)\Phi\left(y\right)$ for $t\in\left[0,1\right]$ \cite{ref9}, $\Phi$ is continuous function, $\Phi\left(0\right) = 0$, and $\lim_{t\to \infty} \Phi \left(t\right)=\infty $.\\

Following \cite{ref7}, suppose $S_{m,N}=\left\{m-N,\cdots,m,\cdots,m+N\right\}$ for some $m\in\mathbb{Z}$ and $N\in\mathbb{N}_0=\mathbb{N}\cup 0$ such that $\left|S_{m,N}\right|=2N+1$ for its cardinality of $S_{m,N}$. Afterwards, let $\phi\in G_\phi$ where $G_\phi$ denotes the set of all functions $\phi : 2\mathbb{N}_0 +1 \rightarrow (0, \infty)$ which $\phi$ fulfill the conditions of being nondecreasing and such that $\frac{\phi\left(2N+1\right)}{2N+1}$ is nonincreasing. Discrete Orlicz-Morrey spaces $\ell_{\phi,\Phi}$ is real sequence spaces $x=\left(x_k\right)_{k\in\mathbb{Z}}$ with norm definded by 
\begin{eqnarray*}
    \left\lVert x\right\rVert _{\ell_{\phi,\Phi}} =\sup_{m\in\mathbb{Z}, N\in\mathbb{N}_0} \left\lVert x\right\rVert _{\phi,\Phi, m, N}
\end{eqnarray*}
where
\begin{eqnarray*}
    \left\lVert x\right\rVert _{\phi,\Phi, m, N} = \inf \left\{b>0 \;\middle|\; \frac{\phi\left(2N+1\right)}{\left|S_{m,N}\right|}\sum_{k\in S_{m,N}}\Phi\left(\frac{\left|x_k\right|}{b}\right) \leq 1\right\}.
\end{eqnarray*}

The article focuses on extending the function used in discrete Orlicz-Morrey spaces by incorporating a continuous \emph{s}-convex function, which takes the value 0 only at the origin. This function is referred to as the \emph{s}-Young function, which is a generalization of the Young function \cite{ref10}. To clarify, the \emph{s}-Young function is defined as a function $\Phi_s : \left[0, \infty\right) \rightarrow \left[0, \infty\right)$ that satisfies $\Phi_s \left(ax+by\right) \leq a^{s}\Phi_s\left(x\right)+b^{s}\Phi_s\left(y\right)$ for $a,b\in\left[0, \infty\right)$, $s\in\left(0,1\right]$, and $a^{s}+b^{s}=1$ . Namely, $\Phi_s$ is the \emph{s}-convex function \cite{ref8}, $\Phi_s$ is continuous function, $\Phi_s \left(0\right) = 0$, and $\lim_{t\to \infty} \Phi_s \left(t\right)=\infty$. The \emph{s}-Young function will be utilized in the construction of the generalized discrete Orlicz-Morrey spaces and can also be called discrete Orlicz-Morrey-\emph{s}. Several properties inherent to discrete Orlicz-Morrey spaces are shown to remain applicable, albeit under modified conditions.\\

The next topic of discussion in this article is the definition and lemmas that are used to obtain the research results. Following this, the results of the study on generalized discrete Orlicz-Morrey spaces are presented, along with the properties suitable to these sequence spaces.

\section{Preliminaries}
This section provides the fundamental concepts and properties such as definitions, lemmas, and properties of the Young function and \emph{s}-Young function, along with the foundational work on discrete Orlicz-Morrey spaces from previous research that underpins the study of generalized discrete Orlicz-Morrey spaces, serving as a basis for the subsequent development of key results.

\begin{definition}
    \cite{ref9} Let $I \subseteq \mathbb{R}$ be an interval. A function $\Phi : I \rightarrow \mathbb{R}$ is said to be convex on $I$ if for any $t\in \left[0,1\right]$ and any points $x,y\in I$ satisfiy
    \begin{eqnarray*}
        \Phi\left(tx+\left(1-t\right)y\right) \leq t\Phi\left(x\right)+\left(1-t\right)\Phi\left(y\right).
    \end{eqnarray*}
\end{definition}

\begin{definition}
    \cite{ref8} Let $s\in\left(0,1\right]$. A function $\Phi_s : \left[0, \infty\right) \rightarrow \mathbb{R}$ is said to be \emph{s}-convex on if for any points $x,y\in \left[0, \infty\right)$ and $a,b\in\left[0, \infty\right)$ satisfy $a^{s}+b^{s}=1$, the inequality
    \begin{eqnarray*}
        \Phi_s \left(ax+by\right) \leq a^{s}\Phi_s\left(x\right)+b^{s}\Phi_s\left(y\right)
    \end{eqnarray*}
    are hold.
\end{definition}

\begin{lemma}
    \cite{ref10} If function $\Phi_s : \left[0, \infty\right) \rightarrow \left[0, \infty\right)$ is a convex function and $\Phi\left(0\right)=0$, $\Phi$ is \emph{s}-convex function.
\end{lemma}

\begin{definition}
    \cite{ref6} A function $\Phi : \left[0, \infty\right) \rightarrow \left[0, \infty\right)$ is said to be a Young function if it satisfies the following conditions:
    \begin{enumerate}
        \item[a)] $\Phi$ is convex function,
        \item[b)] $\Phi$ is continuous function,
        \item[c)] $\Phi\left(0\right) = 0$, and
        \item[d)] $\lim_{t\to \infty} \Phi \left(t\right)=\infty $.
    \end{enumerate}
\end{definition}

\begin{definition}
    \cite{ref10} A function $\Phi_s : \left[0, \infty\right) \rightarrow \left[0, \infty\right)$ is a \emph{s}-Young function if
    \begin{enumerate}
        \item[a)] $\Phi_s$ is \emph{s}-convex function,
        \item[b)] $\Phi_s$ is continuous function,
        \item[c)] $\Phi_s \left(0\right) = 0$, and
        \item[d)] $\lim_{t\to \infty} \Phi_s \left(t\right)=\infty$.
    \end{enumerate}
\end{definition}

\begin{lemma}
    \cite{ref10} Let $\Phi_s : \left[0, \infty\right) \rightarrow \mathbb{R}$ is a \emph{s}-Young function, then
    \begin{enumerate}
        \item[a)] $\Phi_s\left(at\right)\leq a^s \Phi_s\left(t\right)$ for every $t\in\left[0, \infty\right)$ and $0\leq a\leq 1$ with $0<s\leq 1$,
        \item[b)] $\omega\left(t\right)=\frac{\Phi_s\left(t\right)}{t^s}$ is increasing function for every $t\in\left[0, \infty\right)$ with $0<s\leq 1$, and
        \item[c)] $\Phi_s \left(t\right)$ is increasing function for every $t\in\left[0, \infty\right)$.
    \end{enumerate}
\end{lemma}

\begin{definition}
    \cite{ref10} Let $\Phi$ be a \emph{s}-Young function. For every $x \in \left[0, \infty\right)$, the inverse of $\Phi_s$ is defined as:
    \begin{eqnarray*}
        \Phi_s^{-1}(x) = \inf \left\{ r \geq 0 \;\middle|\; \Phi_s(r) > x \right\}.
    \end{eqnarray*}
\end{definition}

\begin{lemma}
    \cite{ref10} If $\Phi$ is a Young’s function, the following properties hold:
    \begin{enumerate}
        \item[a)] $\Phi_s^{-1}(0) = 0$,
        \item[b)] $\Phi_s^{-1}$ is an increasing function, and
        \item[c)] $\Phi_s\left(\Phi_s^{-1}(t)\right) \leq t \leq \Phi_s^{-1}\left(\Phi_s(t)\right)$ for all $t \geq 0$.
    \end{enumerate}
\end{lemma}

\begin{definition}
    \cite{ref7} Let $\Phi : \left[0, \infty\right) \rightarrow \left[0, \infty\right)$ is a Young function, suppose $S_{m,N}=\left\{m-N,\cdots,m,\cdots,m+N\right\}$ for some $m\in\mathbb{Z}$ and $N\in\mathbb{N}_0=\mathbb{N}\cup 0$ such that $\left|S_{m,N}\right|=2N+1$ for its cardinality of $S_{m,N}$. Let $\phi\in G_\phi$ where $G_\phi$ denotes the set of all functions $\phi : 2\mathbb{N}_0 +1 \rightarrow (0, \infty)$ which $\phi$ fulfill the conditions of being nondecreasing and such that $\frac{\phi\left(2N+1\right)}{2N+1}$ is nonincreasing. Discrete Orlicz-Morrey spaces $\ell_{\phi,\Phi}$ is real sequence spaces $x=\left(x_k\right)_{k\in\mathbb{Z}}$ with norm definded by 
    \begin{eqnarray*}
        \left\lVert x\right\rVert _{\ell_{\phi,\Phi}} =\sup_{m\in\mathbb{Z}, N\in\mathbb{N}_0} \left\lVert x\right\rVert _{\phi,\Phi, m, N}
    \end{eqnarray*}
    where
    \begin{eqnarray*}
        \left\lVert x\right\rVert _{\phi,\Phi, m, N} = \inf \left\{b>0 \;\middle|\; \frac{\phi\left(2N+1\right)}{\left|S_{m,N}\right|}\sum_{k\in S_{m,N}}\Phi\left(\frac{\left|x_k\right|}{b}\right) \leq 1\right\}.
    \end{eqnarray*}
\end{definition}

\begin{lemma}
    \cite{ref7} Let $x=\left(x_k\right)$ with $x\in\ell_{\phi,\Phi}$. If  $\left\lVert x\right\rVert _{\phi,\Phi,m,N}\ne 0$, then 
    \begin{eqnarray*}
        \frac{\phi\left(2N+1\right)}{\left|S_{m,N}\right|}\sum_{k\in S_{m,N}}\Phi\left(\frac{\left|x_k\right|}{\left\lVert x\right\rVert _{\phi,\Phi,m,N}}\right)\leq1.
    \end{eqnarray*}
\end{lemma}
 
\begin{lemma}
    \cite{ref7} Let $x=\left(x_k\right)$ with $x\in\ell_{\phi,\Phi}$. $\left\lVert x\right\rVert _{\phi,\Phi,m,N}\leq 1$ if and only if 
    \begin{eqnarray*}
        \frac{\phi\left(2N+1\right)}{\left|S_{m,N}\right|}\sum_{k\in S_{m,N}}\Phi\left(\left|x_k\right|\right)\leq1.
    \end{eqnarray*}
\end{lemma} 

\begin{lemma}
    \cite{ref7} Let $\Phi : \left[0, \infty\right) \rightarrow \left[0, \infty\right)$ be a Young function and $\phi\in G_\phi$. For all $b>0$, $\frac{\phi\left(2N+1\right)}{\left|S_{m,N}\right|}\sum_{k\in S_{m,N}}\Phi\left(\frac{\left|x_k\right|}{b}\right)\leq1$ if and only if $\left\lVert x\right\rVert _{\phi,\Phi,m,N}=0$ for any $x=\left(x_k\right)\in\ell_{\phi,\Phi}$.
\end{lemma}

\begin{lemma}
    \cite{ref7} Let $\Phi : \left[0, \infty\right) \rightarrow \left[0, \infty\right)$ be a Young function, $\phi\in G_\phi$, and $x=\left(x_k\right)$ with $x\in\ell_{\phi,\Phi}$. For each $a>0$, $\frac{\phi\left(2N+1\right)}{\left|S_{m,N}\right|}\sum_{k\in S_{m,N}}\Phi\left(a\left|x_k\right|\right)=0$ if and only if $\left\lVert x\right\rVert _{\phi,\Phi,m,N}=0$.
\end{lemma}

\begin{definition}
    \cite{ref11} The function $\left\lVert \cdot\right\rVert : X \rightarrow \left[0, \infty\right)$ is said to be norm on vector spaces $X$ if for every $x,y\in X$ and $a\in\mathbb{R}$, satisfies
    \begin{enumerate}
        \item[a)] $\left\lVert x\right\rVert\geq 0$,
        \item[b)] $\left\lVert x\right\rVert = 0$ if and only if $x=0$,
        \item[c)] $\left\lVert ax\right\rVert = \left|a\right|\left\lVert x\right\rVert$, and
        \item[d)] $\left\lVert x+y\right\rVert \leq \left\lVert x\right\rVert +\left\lVert y\right\rVert$.
    \end{enumerate}
\end{definition}

\begin{definition}
    \cite{ref12} The function $\left\lVert \cdot\right\rVert : X \rightarrow \left[0, \infty\right)$ is said to be quasi-norm on vector spaces $X$ if for every $x,y\in X$ and $a\in\mathbb{R}$, satisfies
    \begin{enumerate}
        \item[a)] $\left\lVert x\right\rVert\geq 0$,
        \item[b)] $\left\lVert x\right\rVert = 0$ if and only if $x=0$,
        \item[c)] $\left\lVert ax\right\rVert = \left|a\right|\left\lVert x\right\rVert$, and
        \item[d)] There exist $C\geq 1$ such that $\left\lVert x+y\right\rVert \leq C\left(\left\lVert x\right\rVert +\left\lVert y\right\rVert\right)$.
    \end{enumerate}
\end{definition}

\section{Main Results}
\subsection{Definition of Generalized Discrete Orlicz-Morrey Spaces}
Based on the definition of discrete Orlicz-Morrey spaces $\ell_{\phi,\Phi}$ in Definition 2.9, replacing the Young function with an \emph{s}-Young function yields a new formulation of discrete Orlicz-Morrey spaces. Since the \emph{s}-Young function generalizes the Young function, the resulting spaces represent a natural generalization of $\ell_{\phi,\Phi}$. These spaces are referred to as generalized discrete Orlicz-Morrey spaces or discrete Orlicz-Morrey-\emph{s} spaces, denoted by $\ell_{\phi,\Phi_s}$, where $\Phi_s$ is an \emph{s}-Young function.The following outlines the definition and properties of the sequence spaces $\ell_{\phi,\Phi_s}$:
\begin{definition}
    Let $s\in\left(0,1\right]$ and $\Phi_s : \left[0, \infty\right) \rightarrow \left[0, \infty\right)$ is a \emph{s}-Young function, suppose $S_{m,N}=\left\{m-N,\cdots,m,\cdots,m+N\right\}$ for some $m\in\mathbb{Z}$ and $N\in\mathbb{N}_0=\mathbb{N}\cup 0$ such that $\left|S_{m,N}\right|=2N+1$ for its cardinality of $S_{m,N}$. Let $\phi\in G_\phi$ where $G_\phi$ denotes set of all functions $\phi : 2\mathbb{N}_0 +1 \rightarrow (0, \infty)$ which $\phi$ fulfill the conditions of being nondecreasing and such that $\frac{\phi\left(2N+1\right)}{2N+1}$ is nonincreasing. Discrete Orlicz-Morrey spaces $\ell_{\phi,\Phi_s}$ is real sequence spaces $x=\left(x_k\right)_{k\in\mathbb{Z}}$ with a nonnegative function definded by 
    \begin{eqnarray*}
        \left\lVert x\right\rVert _{\ell_{\phi,\Phi_s}} =\sup_{m\in\mathbb{Z}, N\in\mathbb{N}_0} \left\lVert x\right\rVert _{\phi,\Phi_s, m, N}
    \end{eqnarray*}
    where
    \begin{eqnarray*}
        \left\lVert x\right\rVert _{\phi,\Phi_s, m, N} = \inf \left\{b>0 \;\middle|\; \frac{\phi\left(2N+1\right)}{\left|S_{m,N}\right|}\sum_{k\in S_{m,N}}\Phi_s\left(\frac{\left|x_k\right|}{b}\right) \leq 1\right\}.
    \end{eqnarray*}
\end{definition}

Before examining the properties of $\ell_{\phi,\Phi_s}$, an example of an element from the discrete Orlicz-Morrey-\emph{s} space is presented. 

\subsection{Example}
 Consider $\Phi_s\left(t\right) = t^p$ with $t \in \left[0, \infty\right)$ and $p\in\left[1,\infty\right)\cup\left\{\frac{1}{n} \;\middle|\; n\in\mathbb{N}\right\}$, where $\phi\left(2N+1\right) = 2N+1$. For $b > 0$ and $\left(x_k\right) \subseteq \mathbb{R}$, the following can be obtained
\begin{eqnarray*}
    \frac{\phi\left(2N+1\right)}{\left|S_{m,N}\right|}\sum_{k\in S_{m,N}}\Phi_s\left(\frac{\left|x_k\right|}{b}\right)&=&\frac{2N+1}{2N+1}\sum_{k\in S_{m,N}}{\left(\frac{\left|x_k\right|}{b}\right)}^p\\
    &=&\frac{1}{b^p}\sum_{k\in S_{m,N}}{\left|x_k\right|}^p.
\end{eqnarray*}
Let $\left(x_k\right)_{k\in\mathbb{Z}}\subseteq\mathbb{R}$ be the sequence that $x_k=\frac{1}{D^{\frac{\left|k\right|+1}{p}}}$ where $D>1$. Consider $m=0$ and $N\in\mathbb{N}_0$ arbitrary, observe that\\
\begin{eqnarray*}
    \left\lVert x\right\rVert _{\ell_{\phi,\Phi_s}} &=& \sup_{m=0, N \in \mathbb{N}_0} \left(\inf \left\{ b > 0 \;\middle|\; \frac{1}{b^p} \sum_{k \in S_{0,N}} {\left|{\left(\frac{1}{D^{\left|k\right|+1}} \right)}^{\frac{1}{p}}\right|}^{p} \leq 1 \right\} \right)\\
    &=& \sup_{m=0, N \in \mathbb{N}_0} \left(\inf \left\{ b > 0 \;\middle|\; \sum_{k \in S_{0,N}} \left(\frac{1}{D^{\left|k\right|+1}}\right) \leq b^p \right\}\right)\\
    &=& \sup_{m=0, N \in \mathbb{N}_0} \left(\inf \left\{ b > 0 \;\middle|\; \sum_{k=-N}^{-1} \frac{1}{D^{\left|k\right|+1}}+\frac{1}{D^{\left|0\right|+1}}+\sum_{k=1}^{N} \frac{1}{D^{\left|k\right|+1}} \leq b^p \right\}\right)\\
    &=& \sup_{m=0, N \in \mathbb{N}_0} \left(\inf \left\{ b > 0 \;\middle|\;\frac{1}{D^{\left|0\right|+1}}+2\sum_{k=1}^{N} \frac{1}{D^{\left|k\right|+1}} \leq b^p \right\}\right)\\
    &=& \sup_{m=0, N \in \mathbb{N}_0} \left(\inf \left\{ b > 0 \;\middle|\; \frac{1}{D}+2 \cdot \frac{\frac{1}{D^2}\left(1-{\frac{1}{D}}^N\right)}{1-\frac{1}{D}} \leq b^p \right\}\right)\\
    &=& \sup_{m=0, N \in \mathbb{N}_0} \left(\inf \left\{ b > 0 \;\middle|\; \frac{D+1}{D\left(D-1\right)}-\frac{2\cdot{\left(\frac{1}{D}\right)}^N}{D\left(D-1\right)} \leq b^p \right\}\right)\\
    &=& \sup_{m=0, N \in \mathbb{N}_0} \left(\inf \left\{ b > 0 \;\middle|\; \left(\frac{D+1}{D\left(D-1\right)}-\frac{2\cdot{\left(\frac{1}{D}\right)}^N}{D\left(D-1\right)}\right)^{\frac{1}{p}} \leq b \right\}\right)\\
    &=& \sup_{m=0, N \in \mathbb{N}_0} \left(\left(\frac{D+1}{D\left(D-1\right)}-\frac{2\cdot{\left(\frac{1}{D}\right)}^N}{D\left(D-1\right)}\right)^{\frac{1}{p}} \right)\\
    &=& \left(\frac{D+1}{D\left(D-1\right)}\right)^{\frac{1}{p}}.
\end{eqnarray*}

By evaluating the norm of the sequence $x = (x_k)_{k\in\mathbb{Z}} = \left( \frac{1}{D^{\frac{\left|k\right|+1}{p}}} \right)$ where $D>1$, it can be shown that $\left\lVert x\right\rVert _{\ell_{\phi,\Phi_s}} = \left(\frac{D+1}{D\left(D-1\right)}\right)^{\frac{1}{p}} < \infty$. Therefore, $x \in \ell_{\phi,\Phi_s}$.

\subsection{Properties Applicable to Generalized Discrete Orlicz-Morrey Spaces}
After presenting the example of an element in the discrete Orlicz-Morrey-\emph{s} spaces, it is noted by \cite{ref7} that the function $\left\lVert \cdot \right\rVert _{\ell_{\phi,\Phi}}$ in Definition 2.9 defines a norm on the discrete Orlicz-Morrey spaces $\ell_{\phi,\Phi}$. However, the function $\left\lVert \cdot \right\rVert _{\ell_{\phi,\Phi_s}}$ in the discrete Orlicz-Morrey-\emph{s} spaces defines a quasi-norm as it cannot provide the fourth criteria of norm.
\begin{proof}
    With the example condition, consider $x=y=\left( \frac{1}{{\left(1+\sqrt{2}\right)}^{2\left(\left|k\right|+1\right)}} \right)$ and $s=p=\frac{1}{2}$. By the example,it can be obtained $\left\lVert x \right\rVert _{\ell_{\phi, \Phi_s}}=\left\lVert y \right\rVert _{\ell_{\phi, \Phi_s}}=1$. Furthermore, for $x+y=\left( \frac{2}{{\left(1+\sqrt{2}\right)}^{2\left(\left|k\right|+1\right)}} \right)$, it can be obtaned that
    \begin{eqnarray*}
        \left\lVert x+y\right\rVert _{\ell_{\phi,\Phi_s}} &=& \sup_{m=0, N \in \mathbb{N}_0} \left(\inf \left\{ b > 0 \;\middle|\; \frac{1}{b^{\frac{1}{2}}} \sum_{k \in S_{0,N}} {\left|{\left(\frac{2}{\left(1+\sqrt{2}\right)^{\left|k\right|+1}} \right)}^{2}\right|}^{\frac{1}{2}} \leq 1 \right\} \right)\\
        &=& \sup_{m=0, N \in \mathbb{N}_0} \left(\inf \left\{ b > 0 \;\middle|\; \sum_{k \in S_{0,N}} \left(\frac{1}{\left(1+\sqrt{2}\right)^{\left|k\right|+1}}\right) \leq \frac{b^{\frac{1}{2}}}{2} \right\}\right)\\
        &=& \sup_{m=0, N \in \mathbb{N}_0} \left(2^2{\left(1-\frac{2\cdot{\left(\frac{1}{1+\sqrt{2}}\right)}^N}{\left(1+\sqrt{2}\right)\left(1+\sqrt{2}-1\right)}\right)}^2 \right)\\
        &=& 4.
    \end{eqnarray*}
    Because of that, $\left\lVert x+y\right\rVert _{\ell_{\phi,\Phi_s}}=4>1+1=\left\lVert x \right\rVert _{\ell_{\phi, \Phi_s}}+\left\lVert y \right\rVert _{\ell_{\phi, \Phi_s}}$. Therefore, function $\left\lVert \cdot \right\rVert _{\ell_{\phi,\Phi_s}}$ is not a norm.
\end{proof}
Before delving into the quasi-norm, the lemmas supporting its proof are presented.

\begin{lemma}
    Let $x=\left(x_k\right)$ with $x\in\ell_{\phi,\Phi_s}$. If  $\left\lVert x\right\rVert _{\phi,\Phi_s,m,N}\ne 0$, then 
    \begin{eqnarray*}
        \frac{\phi\left(2N+1\right)}{\left|S_{m,N}\right|}\sum_{k\in S_{m,N}}\Phi_s\left(\frac{\left|x_k\right|}{\left\lVert x\right\rVert _{\phi,\Phi_s,m,N}}\right)\leq1.
    \end{eqnarray*}
\end{lemma}
\begin{proof}
    Consider any $x=\left(x_k\right)\in \ell_{\phi, \Phi_s}$ such that $0 < \left\lVert x\right\rVert_{\phi, \Phi_s, m, N} < \infty$, and take any $\varepsilon > 0$. By the property of infimum, there exists $b_1 \in A$ such that $b_1 \leq \left\lVert x\right\rVert_{\phi, \Phi_s, m, N} + \varepsilon$. Therefore, it follows that
    \begin{eqnarray*}
        \frac{\phi\left(2N+1\right)}{\left|S_{m,N}\right|}\sum_{k\in S_{m,N}}\Phi_s\left(\frac{\left|x_k\right|}{\left\lVert x\right\rVert_{\phi, \Phi_s, m, N} + \varepsilon}\right) &\leq& \frac{\phi\left(2N+1\right)}{\left|S_{m,N}\right|}\sum_{k\in S_{m,N}}\Phi_s\left(\frac{\left|x_k\right|}{b_1}\right)\\
        &\leq& 1.
    \end{eqnarray*}
    Since this inequality holds for any $\varepsilon > 0$, it can be concluded that 
    \begin{eqnarray*}
        \frac{\phi\left(2N+1\right)}{\left|S_{m,N}\right|}\sum_{k\in S_{m,N}}\Phi_s\left(\frac{\left|x_k\right|}{\left\lVert x\right\rVert_{\phi, \Phi_s, m, N}}\right)\leq 1.
    \end{eqnarray*}
\end{proof}

\begin{lemma}
    Let $x=\left(x_k\right)$ with $x\in\ell_{\phi,\Phi_s}$. $\left\lVert x\right\rVert _{\phi,\Phi_s,m,N}\leq 1$ if and only if 
    \begin{eqnarray*}
        \frac{\phi\left(2N+1\right)}{\left|S_{m,N}\right|}\sum_{k\in S_{m,N}}\Phi_s\left(\left|x_k\right|\right)\leq1.
    \end{eqnarray*}
\end{lemma}

\begin{proof} Let $x\in\ell_{\phi,\Phi_s}$ arbitrary.
    \begin{enumerate}
        \item [$(\Leftarrow)$] Assume that $\frac{\phi\left(2N+1\right)}{\left|S_{m,N}\right|} \sum_{k \in S_{m,N}} \Phi_s\left(\left|x_k\right|\right) \leq 1$. It will be shown that $\left\lVert x \right\rVert_{\phi, \Phi_s, m, N} \leq 1$.\\
        since
        \begin{eqnarray*}
            \frac{\phi\left(2N+1\right)}{\left|S_{m,N}\right|} \sum_{k \in S_{m,N}} \Phi_s\left(\left|x_k\right|\right) &=&\frac{\phi\left(2N+1\right)}{\left|S_{m,N}\right|} \sum_{k \in S_{m,N}} \Phi_s\left(\frac{\left|x_k\right|}{1}\right)\\
            &\leq& 1,
        \end{eqnarray*}
        this implies that $1 \in A = \left\{b>0 \;\middle|\; \frac{\phi\left(2N+1\right)}{\left|S_{m,N}\right|}\sum_{k\in S_{m,N}}\Phi_s\left(\frac{\left|x_k\right|}{b}\right) \leq 1\right\}$. Based on the definition of the infimum, $\left\lVert x \right\rVert_{\phi, \Phi_s, m, N} \leq 1$.
        \item[($\Rightarrow$)] Assume that $\left\lVert x \right\rVert_{\phi, \Phi_s, m, N} \leq 1$. It will be shown that $\frac{\phi\left(2N+1\right)}{\left|S_{m,N}\right|}\sum_{k \in S_{m,N}} \Phi_s\left(\left|x_k\right|\right) \leq 1$.\\
        Since $\left\lVert x \right\rVert_{\phi, \Phi_s, m, N} \leq 1$, it follows that $\left|x_k\right| \leq \frac{\left|x_k\right|}{\left\lVert x \right\rVert_{\phi, \Phi_s, m, N}}$. Given that $\Phi_s$ is monotonically increasing and $0 < \left\lVert x \right\rVert_{\phi, \Phi_s, m, N} \leq 1$, Lemma 3.3 results in
        \begin{eqnarray*}
            \frac{\phi\left(2N+1\right)}{\left|S_{m,N}\right|} \sum_{k \in S_{m,N}} \Phi_s\left(\left|x_k\right|\right) &\leq& \frac{\phi\left(2N+1\right)}{\left|S_{m,N}\right|} \sum_{k \in S_{m,N}} \Phi_s\left(\frac{\left|x_k\right|}{\left\lVert x \right\rVert_{\phi, \Phi_s, m, N}}\right)\\
            &\leq& 1.
        \end{eqnarray*}
    \end{enumerate}
    It is proven that $\left\lVert x \right\rVert_{\phi, \Phi_s, m, N} \leq 1$ if and only if $\frac{\phi\left(2N+1\right)}{\left|S_{m,N}\right|} \sum_{k \in S_{m,N}} \Phi_s\left(\left|x_k\right|\right) \leq 1$.
\end{proof}

\begin{lemma}
    Let  $\Phi_s : \left[0, \infty\right) \rightarrow \left[0, \infty\right)$ is a Young function and $\phi\in G_\phi$. For all $b>0$, $\frac{\phi\left(2N+1\right)}{\left|S_{m,N}\right|}\sum_{k\in S_{m,N}}\Phi_s\left(\frac{\left|x_k\right|}{b}\right)\leq1$ if and only if $\left\lVert x\right\rVert _{\phi,\Phi_s,m,N}=0$ for any $x=\left(x_k\right)\in\ell_{\phi,\Phi_s}$.
\end{lemma}

\begin{proof}
    Let  $\Phi_s : \left[0, \infty\right) \rightarrow \left[0, \infty\right)$ is a Young function and $\phi\in G_\phi$.
    \begin{enumerate}
        \item [$(\Leftarrow)$] Assume that for every $b > 0$, the inequality $\frac{\phi(2N+1)}{\left|S_{m,N}\right|} \sum_{k \in S_{m,N}} \Phi_s\left(\frac{\left|x_k\right|}{b}\right) \leq 1$ holds. It will be shown that $\left\lVert x \right\rVert_{\phi, \Phi_s, m, N} = 0$ for every $x = \left(x_k\right) \in \ell_{\phi, \Phi_s}$.\\
        Since 
        \begin{eqnarray*}
            \frac{\phi(2N+1)}{\left|S_{m,N}\right|} \sum_{k \in S_{m,N}} \Phi_s\left(\frac{\left|x_k\right|}{\varepsilon}\right) \leq 1
        \end{eqnarray*}
        for every $\varepsilon > 0$, it follows that $\varepsilon \in \left\{b > 0 \;\middle|\; \frac{\phi(2N+1)}{\left|S_{m,N}\right|} \sum_{k \in S_{m,N}} \Phi_s\left(\frac{\left|x_k\right|}{b}\right) \leq 1\right\}$, which implies $0 \leq \left\lVert x \right\rVert_{\phi, \Phi_s, m, N} \leq \varepsilon$. Therefore, it can be concluded that $\left\lVert x \right\rVert_{\phi, \Phi_s, m, N} = 0$.
        
        \item[($\Rightarrow$)] Assume that $\left\lVert x \right\rVert_{\phi, \Phi_s, m, N} = 0$ for every $x = \left(x_k\right) \in \ell_{\phi, \Phi_s}$. It will be shown that 
        \begin{eqnarray*}
            \frac{\phi(2N+1)}{\left|S_{m,N}\right|} \sum_{k \in S_{m,N}} \Phi_s\left(\frac{\left|x_k\right|}{b}\right) \leq 1
        \end{eqnarray*}
        for every $b > 0$.\\
        Let $A = \left\{b > 0 \;\middle|\; \frac{\phi(2N+1)}{\left|S_{m,N}\right|} \sum_{k \in S_{m,N}} \Phi_s\left(\frac{\left|x_k\right|}{b}\right) \leq 1\right\}.$ To demonstrate this by contradiction, assume there exists $\varepsilon_0 > 0$ such that $\frac{\phi(2N+1)}{\left|S_{m,N}\right|} \sum_{k \in S_{m,N}} \Phi_s\left(\frac{\left|x_k\right|}{\varepsilon_0}\right) > 1.$ This implies $\varepsilon_0 \notin A$. Next, take any $b_1 \in A$ such that $b_1 \neq \varepsilon_0$ and $\left\lVert x \right\rVert_{\phi, \Phi_s, m, N} \leq b_1$. Consider the following two cases:\\
        Case 1: If $b_1 < \varepsilon_0$, then $\frac{1}{\varepsilon_0} < \frac{1}{b_1}$. Consequently, it follows that
        \begin{eqnarray*}
            \frac{\phi(2N+1)}{\left|S_{m,N}\right|} \sum_{k \in S_{m,N}} \Phi_s\left(\frac{\left|x_k\right|}{b}\right) < \frac{\phi(2N+1)}{\left|S_{m,N}\right|} \sum_{k \in S_{m,N}} \Phi_s\left(\frac{\left|x_k\right|}{b_1}\right) \leq 1.
        \end{eqnarray*}
        This contradicts the assumption that $\frac{\phi(2N+1)}{\left|S_{m,N}\right|} \sum_{k \in S_{m,N}} \Phi_s\left(\frac{\left|x_k\right|}{b}\right) > 1.$\\
        
        Case 2: If $b_1 > \varepsilon_0$, by the definition of the infimum, $\left\lVert x \right\rVert_{\phi, \Phi_s, m, N} \geq \varepsilon_0 > 0$. This contradicts the assumption that $\left\lVert x \right\rVert_{\phi, \Phi_s, m, N} = 0$.
    \end{enumerate}
    From both cases, it follows that $\frac{\phi(2N+1)}{\left|S_{m,N}\right|} \sum_{k \in S_{m,N}} \Phi_s\left(\frac{\left|x_k\right|}{\varepsilon}\right) \leq 1$ for every $\varepsilon > 0$.
    It is proven that for every $b > 0$, $ \frac{\phi(2N+1)}{\left|S_{m,N}\right|} \sum_{k \in S_{m,N}} \Phi_s\left(\frac{\left|x_k\right|}{b}\right) \leq 1$ if and only if $\left\lVert x \right\rVert_{\phi, \Phi_s, m, N} = 0$ for every $x = \left(x_k\right) \in \ell_{\phi, \Phi_s}$.
\end{proof}

\begin{lemma}
    Let  $\Phi_s : \left[0, \infty\right) \rightarrow \left[0, \infty\right)$ is a Young function and $\phi\in G_\phi$. For all $b \geq 1$, $\frac{\phi\left(2N+1\right)}{\left|S_{m,N}\right|} \sum_{k \in S_{m,N}} \Phi_s\left(\frac{\left|x_k\right|}{b}\right) \leq \frac{1}{b^s}$ if and only if $\left\lVert x \right\rVert_{\phi,\Phi_s,m,N} \leq 1$ for every $x = \left(x_k\right) \in \ell_{\phi,\Phi_s}$.
\end{lemma}

\begin{proof} Let  $\Phi_s : \left[0, \infty\right) \rightarrow \left[0, \infty\right)$ is a Young function and $\phi\in G_\phi$.
    \begin{enumerate} 
        \item [$(\Leftarrow)$] Assume that for every $b \geq 1$, the inequality $ \frac{\phi\left(2N+1\right)}{\left|S_{m,N}\right|} \sum_{k \in S_{m,N}} \Phi_s\left(\frac{\left|x_k\right|}{b}\right) \leq \frac{1}{b^s}$ holds. It will be shown that $\left\lVert x \right\rVert_{\phi,\Phi_s,m,N} \leq 1$ for every $x = \left(x_k\right) \in \ell_{\phi,\Phi_s}$.\\
        Since $b \geq 1$ and $0 < s \leq 1$, it follows that $0 < \frac{1}{b} \leq \frac{1}{b^s} \leq 1$. This implies
        \begin{eqnarray*}
            \frac{\phi\left(2N+1\right)}{\left|S_{m,N}\right|} \sum_{k \in S_{m,N}} \Phi_s\left(\frac{\left|x_k\right|}{b}\right) \leq \frac{1}{b^s} \leq 1
        \end{eqnarray*}
        for every $b \geq 1$. By choosing $b = 1$, it follows that 
        \begin{eqnarray*}
            \frac{\phi\left(2N+1\right)}{\left|S_{m,N}\right|} \sum_{k \in S_{m,N}} \Phi_s\left(\frac{\left|x_k\right|}{1}\right) \leq 1,
        \end{eqnarray*}
        which means $1 \in \left\{b > 0 \;\middle|\; \frac{\phi\left(2N+1\right)}{\left|S_{m,N}\right|} \sum_{k \in S_{m,N}} \Phi_s\left(\frac{\left|x_k\right|}{b}\right) \leq 1\right\}$. By the definition of the infimum, $\left\lVert x \right\rVert_{\phi,\Phi_s,m,N} \leq 1$.
    
        \item[($\Rightarrow$)] Assume that $\left\lVert x \right\rVert_{\phi,\Phi_s,m,N} \leq 1$ for every $x = \left(x_k\right) \in \ell_{\phi,\Phi_s}$. It will be shown that 
        \begin{eqnarray*}
            \frac{\phi\left(2N+1\right)}{\left|S_{m,N}\right|} \sum_{k \in S_{m,N}} \Phi_s\left(\frac{\left|x_k\right|}{b}\right) \leq \frac{1}{b^s}
        \end{eqnarray*}
        for every $b \geq 1$.\\
        Since $\left\lVert x \right\rVert_{\phi,\Phi_s,m,N} \leq 1$, it follows that $\left|x_k\right| \leq \frac{\left|x_k\right|}{\left\lVert x \right\rVert_{\phi,\Phi_s,m,N}}$. Given $0 < \frac{1}{b} \leq 1$ and $0 < \left\lVert x \right\rVert_{\phi,\Phi_s,m,N} \leq 1$, and based on Lemma 2.6.(a) and Lemma 3.2, it follows that 
        \begin{eqnarray*}
            \frac{\phi\left(2N+1\right)}{\left|S_{m,N}\right|} \sum_{k \in S_{m,N}} \Phi_s\left(\frac{\left|x_k\right|}{b}\right) &\leq& \frac{1}{b^s} \frac{\phi\left(2N+1\right)}{\left|S_{m,N}\right|} \sum_{k \in S_{m,N}} \Phi_s\left(\left|x_k\right|\right) \\   &\leq& \frac{1}{b^s} \frac{\phi\left(2N+1\right)}{\left|S_{m,N}\right|} \sum_{k \in S_{m,N}} \Phi_s\left(\frac{\left|x_k\right|}{\left\lVert x \right\rVert_{\phi,\Phi_s,m,N}}\right) \\
            &\leq& \frac{1}{b^s}.
        \end{eqnarray*}
    \end{enumerate}
     Thus, it is proven that for every $b \geq 1$, $\frac{\phi\left(2N+1\right)}{\left|S_{m,N}\right|} \sum_{k \in S_{m,N}} \Phi_s\left(\frac{\left|x_k\right|}{b}\right) \leq \frac{1}{b^s}$ if and only if $\left\lVert x \right\rVert_{\phi,\Phi_s,m,N} \leq 1$ for every $x = \left(x_k\right) \in \ell_{\phi,\Phi_s}$.
\end{proof}

\begin{lemma}
    Let $\Phi_s : \left[0, \infty\right) \rightarrow \left[0, \infty\right)$ is a Young function,  $\phi\in G_\phi$, and $x=\left(x_k\right)$ with $x\in\ell_{\phi,\Phi_s}$. For each $a>0$, $\frac{\phi\left(2N+1\right)}{\left|S_{m,N}\right|}\sum_{k\in S_{m,N}}\Phi_s\left(a\left|x_k\right|\right)=0$ if and only if $\left\lVert x\right\rVert _{\phi,\Phi_s,m,N}=0$.
\end{lemma}

\begin{proof}
    \begin{enumerate}
        \item [$(\Leftarrow)$] Assume that for every $a > 0$, the equation $\frac{\phi\left(2N+1\right)}{\left|S_{m,N}\right|} \sum_{k \in S_{m,N}} \Phi_s\left(a\left|x_k\right|\right) = 0$ holds. It will be shown that $\left\lVert x \right\rVert_{\phi,\Phi_s,m,N} = 0$.
        Let $a > 0$ be arbitrary. Note that 
        \begin{eqnarray*}
            \frac{\phi\left(2N+1\right)}{\left|S_{m,N}\right|} \sum_{k \in S_{m,N}} \Phi_s\left(a\left|x_k\right|\right) = 0 \leq 1.
        \end{eqnarray*}
        This implies 
        \begin{eqnarray*}
            \frac{\phi\left(2N+1\right)}{\left|S_{m,N}\right|} \sum_{k \in S_{m,N}} \Phi_s\left(\frac{\left|x_k\right|}{\frac{1}{a}}\right) \leq 1,
        \end{eqnarray*}
        so that $\frac{1}{a} \in \left\{b > 0 \;\middle|\; \frac{\phi\left(2N+1\right)}{\left|S_{m,N}\right|} \sum_{k \in S_{m,N}} \Phi_s\left(\frac{\left|x_k\right|}{b}\right) \leq 1\right\}$. By the definition of the infimum, $0 \leq \left\lVert x \right\rVert_{\phi,\Phi_s,m,N} \leq \frac{1}{a}$ for every $a > 0$. Consequently, $\left\lVert x \right\rVert_{\phi,\Phi_s,m,N} = 0$.
    
        \item[($\Rightarrow$)] Assume that $\left\lVert x \right\rVert_{\phi,\Phi_s,m,N} = 0$. It will be shown that $\frac{\phi\left(2N+1\right)}{\left|S_{m,N}\right|} \sum_{k \in S_{m,N}} \Phi_s\left(a\left|x_k\right|\right) = 0$ for every $a > 0$\\
        Let $0 < \varepsilon < 1$ be arbitrary. Using the Lemma 2.6.(a) and Lemma 3.4, it follows that 
        \begin{eqnarray*}
            \Phi_s\left(a\left|x_k\right|\right) = \Phi_s\left(\varepsilon\left(\frac{a\left|x_k\right|}{\varepsilon}\right)\right) \leq \varepsilon^s \Phi_s\left(\frac{a\left|x_k\right|}{\varepsilon}\right) \leq \varepsilon^s \cdot 1 = \varepsilon^s.
        \end{eqnarray*}
        Since $0 \leq \Phi_s\left(a\left|x_k\right|\right) \leq \varepsilon^s$ for every $0 < \varepsilon^s < 1$, it follows that $\Phi_s\left(a\left|x_k\right|\right) = 0$. 
    \end{enumerate}
    Thus, it has been proven that for every $a > 0$, $\frac{\phi\left(2N+1\right)}{\left|S_{m,N}\right|} \sum_{k \in S_{m,N}} \Phi_s\left(a\left|x_k\right|\right) = 0$ if and only if $\left\lVert x \right\rVert_{\phi,\Phi_s,m,N} = 0$.
\end{proof}

After showing the Lemmas, The distinction between $\left\lVert x \right\rVert _{\ell_{\phi,\Phi}}$ and $\left\lVert x \right\rVert _{\ell_{\phi,\Phi_s}}$ is further clarified in next theorem.

\begin{theorem}
    The nonnegative function $\left\lVert \cdot \right\rVert _{\ell_{\phi,\Phi_s}}$ is quasi-norm.
\end{theorem}

\begin{proof}
    Consider any sequences $x=\left(x_k\right)$ and $y=\left(y_k\right)$ where $x, y \in \ell_{\phi, \Phi_s}$. To demonstrate that $\left\lVert x \right\rVert _{\ell_{\phi,\Phi_s}}$ is a quasi-norm function, it is necessary to verify the following four conditions based on Definition 2.15:

    \begin{enumerate}
        \item It will be shown that $\left\lVert x \right\rVert _{\ell_{\phi,\Phi_s}}\geq 0$.
        Let $A = \left\{b>0 \;\middle|\; \frac{\phi\left(2N+1\right)}{\left|S_{m,N}\right|}\sum_{k\in S_{m,N}}\Phi_s\left(\frac{\left|x_k\right|}{b}\right) \leq 1\right\}$. Since $\left\lVert x \right\rVert _{\phi, \Phi_s, m, N}=\inf A$, by the definition of the infimum, $0 \leq \|x\|_{\phi, \Phi_s, m, N} \leq b$ for every $b \in A$. Thus, $\left\lVert x \right\rVert _{\phi, \Phi_s, m, N}\geq 0$. Therefore, by the definition of $\left\lVert \cdot \right\rVert _{\ell_{\phi,\Phi_s}}$,
        \begin{eqnarray*}
            \left\lVert x \right\rVert _{\ell_{\phi,\Phi_s}}=\sup_{m \in \mathbb{Z}, N \in \mathbb{N}_0} \left\lVert x \right\rVert _{\phi, \Phi_s, m, N}\geq 0
        \end{eqnarray*}
        Hence, it is proved that $\left\lVert x \right\rVert _{\ell_{\phi,\Phi_s}}\geq 0$.
        
        \item It will be shown that $\left\lVert x \right\rVert _{\ell_{\phi,\Phi_s}}=0$ if and only if $x=0$.
        \begin{enumerate}
            \item [$(\Leftarrow)$] Assume $x = \left(x_k\right) = 0$. It will be demonstrated that $\left\lVert x \right\rVert _{\ell_{\phi,\Phi_s}}=0$. Take any $b>0$, $m \in \mathbb{Z}$, and $N \in \mathbb{N}_0$, observe that 
            \begin{eqnarray*}
                \frac{\phi\left(2N+1\right)}{\left|S_{m,N}\right|}\sum_{k\in S_{m,N}}\Phi_s\left(\frac{\left|x_k\right|}{b}\right)=0 \leq 1
            \end{eqnarray*}
            Therefore, $b \in A$, and it follows that $\left\lVert x \right\rVert _{\phi, \Phi_s, m, N}\leq b$ for every $b>0$, implying $\left\lVert x \right\rVert _{\phi, \Phi_s, m, N}=0$. By definition of $\left\lVert x \right\rVert _{\ell_{\phi,\Phi_s}}$,
            \begin{eqnarray*}
                \left\lVert x \right\rVert _{\ell_{\phi,\Phi_s}}&=&\sup_{m \in \mathbb{Z}, N \in \mathbb{N}_0} \left\lVert x \right\rVert _{\phi, \Phi_s, m, N}\\
                   &=&0.
            \end{eqnarray*}
               
            \item[($\Rightarrow$)] Assume $\left\lVert x \right\rVert _{\ell_{\phi,\Phi_s}}=0$. Observe that $\left\lVert x\right\rVert _{\ell_{\phi,\Phi_s}} =\sup_{m\in\mathbb{Z}, N\in\mathbb{N}_0} \left\lVert x\right\rVert _{\phi,\Phi_s, m, N}=0$. It means $\left\lVert x\right\rVert _{\phi,\Phi_s, m, N}=0$. By Lemma 3.6, it can be obtained that $\frac{\phi\left(2N+1\right)}{\left|S_{m,N}\right|}\sum_{k\in S_{m,N}}\Phi_s\left(a\left|x_k\right|\right)=0$. Therefore, it must be that $\Phi_s\left(a\left|x_k\right|\right) = 0$. Consequently, $x_k = 0$ for all $k \in S_{m,N}$, which implies $x = \left(x_k\right)_{k\in \mathbb{Z}} = 0$. 
            It is proven that $\left\lVert x \right\rVert _{\ell_{\phi,\Phi_s}}=0$ if and only if $x=0$..
        \end{enumerate}
        
        \item It will be shown that $\left\lVert ax \right\rVert _{\ell_{\phi,\Phi_s}}=\left|a\right|\cdot\left\lVert x \right\rVert _{\ell_{\phi,\Phi_s}}$ for every $a \in \mathbb{R}$.
        
        Case 1: $a = 0$.\\
        Observe that
        \begin{eqnarray*}
            \left\lVert ax \right\rVert _{\ell_{\phi,\Phi_s}}&=&\left\lVert 0\cdot x \right\rVert _{\ell_{\phi,\Phi_s}}\\
            &=&\left\lVert 0\cdot x \right\rVert _{\ell_{\phi,\Phi_s}}\\
            &=&0\\
            &=&\left|0\right|\cdot\left\lVert x \right\rVert _{\ell_{\phi,\Phi_s}}.
        \end{eqnarray*}
        case 2: $a \neq 0$.\\
        Take any $a \in \mathbb{R} \setminus \left\{0\right\}$. Note that
        \begin{eqnarray*}
            \left\lVert ax \right\rVert _{\phi, \Phi_s, m, N}&=&\inf \left\{b > 0 \;\middle|\; \frac{\phi\left(2N+1\right)}{\left|S_{m,N}\right|}\sum_{k\in S_{m,N}}\Phi_s\left(\frac{\left|a\cdot x_k\right|}{b}\right) \leq 1\right\}\\
            &=&\inf \left\{b > 0 \;\middle|\; \frac{\phi\left(2N+1\right)}{\left|S_{m,N}\right|}\sum_{k\in S_{m,N}}\Phi_s\left(\frac{\left|a\right|\cdot\left|x_k\right|}{b}\right) \leq 1\right\}\\
            &=&\inf \left\{b > 0 \;\middle|\; \frac{\phi\left(2N+1\right)}{\left|S_{m,N}\right|}\sum_{k\in S_{m,N}}\Phi_s\left(\frac{\left|x_k\right|}{\frac{b}{\left|a\right|}}\right) \leq 1\right\}\\
        \end{eqnarray*}
        Using the substitution $c = \frac{b}{\left|a\right|}$, it follows that
        \begin{eqnarray*}
            \left\lVert ax \right\rVert _{\phi, \Phi_s, m, N}&=&\inf \left\{c\left|a\right|> 0 \;\middle|\; \frac{\phi\left(2N+1\right)}{\left|S_{m,N}\right|}\sum_{k\in S_{m,N}}\Phi_s\left(\frac{\left|x_k\right|}{c}\right) \leq 1\right\}\\
            &=&\left|a\right|\inf \left\{c > 0 \;\middle|\; \frac{\phi\left(2N+1\right)}{\left|S_{m,N}\right|}\sum_{k\in S_{m,N}}\Phi_s\left(\frac{\left|x_k\right|}{c}\right) \leq 1\right\}\\
            &=&\left|a\right|\cdot\left\lVert x \right\rVert _{\phi, \Phi_s, m, N}.
        \end{eqnarray*}
        Thus,
        \begin{eqnarray*}
            \left\lVert ax \right\rVert _{\ell_{\phi,\Phi_s}}&=&\sup_{m \in \mathbb{Z}, N \in \mathbb{N}_0} \left\lVert a x\right\rVert_{\phi, \Phi_s, m, N}\\
            &=& \left|a\right|\cdot\sup_{m \in \mathbb{Z}, N \in \mathbb{N}_0} \left\lVert x\right\lVert_{\phi, \Phi_s, m, N}\\ 
            &=& \left| a\right| \left\lVert x\right\lVert_{\ell_{\phi, \Phi_s}}.
        \end{eqnarray*}
        Therefore, it is proven that $\left\lVert ax \right\rVert _{\ell_{\phi,\Phi_s}}=\left| a\right| \left\lVert x\right\lVert_{\ell_{\phi, \Phi_s}}$. for every $a\in \mathbb{R}$.
        
        \item It will be shown that there exists $C\geq 1$ such that
        \begin{eqnarray*}
            \left\lVert x+y \right\rVert _{\ell_{\phi,\Phi_s}}\leq C^{\frac{1}{s}} \left(\left\lVert x \right\rVert _{\ell_{\phi, \Phi_s}} + \left\lVert y \right\rVert _{\ell_{\phi, \Phi_s}} \right). \forall s\in \left(0,1\right].
        \end{eqnarray*}
        Case 1: $x=0$ and $y=0$.\\
        Since $x=0$ and $y=0$, it follows that $\left\lVert x \right\rVert _{\ell_{\phi,\Phi_s}}=0$ and $\left\lVert y \right\rVert _{\ell_{\phi,\Phi_s}}=0$.
        Therefore,
        \begin{eqnarray*}
            \left\lVert x+y \right\rVert _{\ell_{\phi,\Phi_s}} &=&\left\lVert 0 \right\rVert _{\ell_{\phi,\Phi_s}}\\
            &=& 0\\
            &=& C^{\frac{1}{s}}\left(0+0\right)\\
            &=& C^{\frac{1}{s}}\left(\left\lVert x \right\rVert _{\ell_{\phi,\Phi_s}}+\left\lVert y \right\rVert _{\ell_{\phi,\Phi_s}}\right). 
        \end{eqnarray*}
        Case 2: $x\neq0$ and $y=0$ or vise versa.\\  
        Without loss of generality, assume $x\neq0$ and $y=0$. Choose $C=1$, so 
        \begin{eqnarray*}
            \left\lVert x+y \right\rVert _{\ell_{\phi,\Phi_s}} &=&\left\lVert x+0 \right\rVert _{\ell_{\phi,\Phi_s}}\\
            &=& 1\left(\left\lVert x \right\rVert _{\ell_{\phi,\Phi_s}}+0\right)\\ &=&C^{\frac{1}{s}}\left(\left\lVert x \right\rVert _{\ell_{\phi,\Phi_s}}+\left\lVert y \right\rVert _{\ell_{\phi,\Phi_s}}\right).
        \end{eqnarray*}
        Case 3: $x\neq 0$, $y\neq 0$.\\
        Let $B = \left\{b > 0 \;\middle|\; \frac{\phi\left(2N+1\right)}{\left|S_{m,N}\right|}\sum_{k\in S_{m,N}} \Phi_s \left(\frac{\left|x_k + y_k\right|}{b}\right) \leq 1\right\}$, $X=\left\lVert x \right\rVert _{\phi, \Phi_s, m, N}$, and $Y=\left\lVert y\right\rVert _{\phi, \Phi_s, m, N}$.
        Define $C = \frac{X^s + Y^s}{(X + Y)^s}$. Observe that
        \begin{eqnarray*}
            \frac{\phi\left(2N+1\right)}{\left|S_{m,N}\right|}\sum_{k\in S_{m,N}}\Phi_s\left(\frac{\left|x_k+y_k\right|}{C^{\frac{1}{s}} \left(X+Y\right)}\right) &\leq  \frac{\phi\left(2N+1\right)}{\left|S_{m,N}\right|}\sum_{k\in S_{m,N}}\Phi_s\left(\frac{\left|x_k\right|+\left|y_k\right|}{C^{\frac{1}{s}} \left(X+Y\right)}\right)\\
            &=  \frac{\phi\left(2N+1\right)}{\left|S_{m,N}\right|}\sum_{k\in S_{m,N}}\Phi_s\left(\frac{\left|x_k\right|}{C^{\frac{1}{s}} \left(X+Y\right)}\frac{X}{X} \right.\\
            & \left.+\frac{\left|y_k\right|}{C^{\frac{1}{s}} \left(X+Y\right)}\frac{Y}{Y}\right). 
        \end{eqnarray*}
        Because of $\left(\frac{X}{C^{\frac{1}{s}} \left(X+Y\right)}\right)^s+\left(\frac{Y}{C^{\frac{1}{s}} \left(X+Y\right)}\right)^s=1$, by \emph{s}-convex definition,
        \begin{multline*}
            \frac{\phi(2N+1)}{|S_{m,N}|} \sum_{k\in S_{m,N}} 
            \Phi_s\!\left(
               \frac{|x_k|}{C^{1/s}(X+Y)}\frac{X}{X}
               + \frac{|y_k|}{C^{1/s}(X+Y)}\frac{Y}{Y}
            \right)
            \\
            \leq
            \left(\frac{X}{C^{1/s}(X+Y)}\right)^s 
               \frac{\phi(2N+1)}{|S_{m,N}|}\sum_{k\in S_{m,N}} \Phi_s\!\left(\frac{|x_k|}{X}\right)
            \\
            + \left(\frac{Y}{C^{1/s}(X+Y)}\right)^s 
               \frac{\phi(2N+1)}{|S_{m,N}|}\sum_{k\in S_{m,N}} \Phi_s\!\left(\frac{|y_k|}{Y}\right)
            \\
            \leq
            \left(\frac{X}{C^{1/s}(X+Y)}\right)^s
            + \left(\frac{Y}{C^{1/s}(X+Y)}\right)^s
            = 1.
        \end{multline*}

        It means $C^{\frac{1}{s}}\left(X+Y\right)=C^{\frac{1}{s}}\left(\left\lVert x \right\rVert _{\phi, \Phi_s, m, N}+\left\lVert y \right\rVert _{\phi, \Phi_s, m, N}\right)\in B$. Since $\left\lVert x+y \right\rVert _{\phi, \Phi_s, m, N}$ is infimum of $B$, then $\left\lVert x+y \right\rVert _{\phi, \Phi_s, m, N}\leq C^{\frac{1}{s}}\left(X+Y\right)=C^{\frac{1}{s}}\left(\left\lVert x \right\rVert _{\phi, \Phi_s, m, N}+\left\lVert y \right\rVert _{\phi, \Phi_s, m, N}\right)$.
        By taking the supremum over $m \in \mathbb{Z}$ and $N \in \mathbb{N}_0$, it can be obtained that $\left\lVert x+y \right\rVert _{\ell_{\phi,\Phi_s}}\leq C^{\frac{1}{s}} \left(\left\lVert x \right\rVert _{\ell_{\phi, \Phi_s}} + \left\lVert y \right\rVert _{\ell_{\phi, \Phi_s}} \right)$.
        Furthermore, it will be demonstrated that $1 \leq C < 2$. For every $s \in \left(0,1\right]$, observe that
        \begin{eqnarray*}
            \frac{X}{X + Y} \leq \left(\frac{X}{X + Y}\right)^s < 1
        \end{eqnarray*}
        and
        \begin{eqnarray*}
            \frac{Y}{X + Y} \leq \left(\frac{Y}{X + Y}\right)^s < 1.
        \end{eqnarray*}
        As a result,
        \begin{eqnarray*}
            1 = \frac{X + Y}{X + Y} \leq \frac{X^s + Y^s}{(X + Y)^s} = C < 1 + 1 = 2.
        \end{eqnarray*}
        Thus, it has been established that \(1 \leq C < 2\), ensuring that $\left\lVert x+y \right\rVert _{\ell_{\phi,\Phi_s}}\leq C^{\frac{1}{s}} \left(\left\lVert x \right\rVert _{\ell_{\phi, \Phi_s}} + \left\lVert y \right\rVert _{\ell_{\phi, \Phi_s}} \right). \forall s\in \left(0,1\right]$.  
    \end{enumerate} 
    Since conditions (1) through (4) are satisfied, it is established that $\left\lVert x \right\rVert _{\ell_{\phi, \Phi_s}}$ is a quasi-norm on the Orlicz-Morrey-s sequence space $\ell_{\phi, \Phi_s}$.
\end{proof}

After that, the completeness of the Orlicz-Morrey-s sequence space $\ell_{\phi, \Phi_s}$ will be demonstrated. Furthermore, the space $\left(\ell_{\phi, \Phi_s}, \left\lVert \cdot \right\rVert_{\ell_{\phi, \Phi_s}}\right)$ is shown to be a quasi-Banach space. This result is presented in the following theorem.

\begin{theorem}
    The spaces $\left(\ell_{\phi, \Phi_s}, \left\lVert \cdot \right\rVert_{\ell_{\phi, \Phi_s}}\right)$ is a quasi-Banach space.
\end{theorem}

\begin{proof}
    By Theorem 3.7, it has been established that $\left\lVert \cdot \right\rVert_{\ell_{\phi, \Phi_s}}$ is a quasi-norm. Therefore, it suffices to show that $\ell_{\phi, \Phi_s}$ is complete by proving that any Cauchy sequence in $\ell_{\phi, \Phi_s}$ converges to an element in $\ell_{\phi, \Phi_s}$.
    
    Consider an arbitrary Cauchy sequence $\left(x_k\right)_{k \in \mathbb{Z}} \subset \ell_{\phi, \Phi_s}$ where $x_k = \left(x_j^{(k)}\right)_{j \in \mathbb{Z}} = \left(\ldots, x_{-1}^{(k)}, x_0^{(k)}, x_1^{(k)}, \ldots\right)$ and $x_j^{(k)} \in \ell_{\phi, \Phi_s}$. Consequently, for any $\varepsilon > 0$, there exists $K_\varepsilon \in \mathbb{N}$ such that for all $m, n \geq K_\varepsilon$, the inequality
    \begin{eqnarray*}
        \left\lVert x_n - x_m \right\rVert_{\ell_{\phi, \Phi_s}} < \varepsilon
    \end{eqnarray*}
    holds. Based on the definition of $\left\lVert \cdot \right\rVert_{\ell_{\phi, \Phi_s}}$, it follows that
    \begin{eqnarray*}
        \left\lVert x_n - x_m \right\rVert_{\phi, \Phi_s, m, N} \leq \left\lVert x_n - x_m \right\rVert_{\ell_{\phi, \Phi_s}} < \varepsilon.
    \end{eqnarray*}
    Using Lemma 3.4, it is obtained that
    \begin{eqnarray*}
        \frac{\phi\left(2N+1\right)}{\left|S_{m, N}\right|} \sum_{j \in S_{m, N}} \Phi_s\left(\frac{\left|x_j^{(n)} - x_j^{(m)}\right|}{\varepsilon}\right) \leq 1.
    \end{eqnarray*}
    Choosing $N = 0$, it follows that 
    \begin{eqnarray*}
        \phi(1) \sum_{j \in S_{m, N}} \Phi_s\left(\frac{\left|x_j^{(n)} - x_j^{(m)}\right|}{\varepsilon}\right) \leq 1.
    \end{eqnarray*}
    As a result, 
    \begin{eqnarray*}
        \Phi_s\left(\frac{\left|x_j^{(n)} - x_j^{(m)}\right|}{\varepsilon}\right) \leq \frac{1}{\phi(1)}.
    \end{eqnarray*}
    By Lemmas 2.8(c) and 2.6(c), it is obtained that
    \begin{eqnarray*}
        \frac{\left|x_j^{(n)} - x_j^{(m)}\right|}{\varepsilon} \leq \Phi_s^{-1}\left(\Phi_s\left(\frac{\left|x_j^{(n)} - x_j^{(m)}\right|}{\varepsilon}\right)\right) \leq \Phi_s^{-1}\left(\frac{1}{\phi(1)}\right).
    \end{eqnarray*}
    This implies $\left|x_j^{(n)} - x_j^{(m)}\right| < \varepsilon \Phi_s^{-1}\left(\frac{1}{\phi(1)}\right)$, so for a fixed $j$, the sequence $x_k^* = \left(x_j^{(k)}\right)_{k \in \mathbb{Z}} = \left(\ldots, x_j^{(-1)}, x_j^{(0)}, x_j^{(1)}, \ldots\right)$ forms a Cauchy sequence in $\mathbb{R}$. By the Cauchy convergence criterion, $x_k^* = \left(x_j^{(k)}\right)_{k \in \mathbb{Z}}$ converges, and there exists $x_j' \in \mathbb{R}$ such that
    \begin{eqnarray*}
        \lim_{k \to \infty} x_k^* = \lim_{k \to \infty} x_j^{(k)} = x_j'.
    \end{eqnarray*}
    Let $x' = \left(x_j'\right) = \left(\ldots, x_{-1}', x_0', x_1', \ldots\right)$. It will be demonstrated that $x' \in \ell_{\phi, \Phi_s}$ and $\lim x_k = \lim_{j \to \infty} x_j^{(k)} = x_j'$. Since $\left\lVert x_n - x_m \right\rVert_{\ell_{\phi, \Phi_s}} < \varepsilon$ and $\lim_{m \to \infty} x_j^{(m)} = x_j'$, it follows that
    \begin{eqnarray*}
        \left\lVert x_n - x' \right\rVert_{\ell_{\phi, \Phi_s}} < \varepsilon < \infty.
    \end{eqnarray*}
    This implies $x_n - x' \in \ell_{\phi, \Phi_s}$. From the quasi-norm definition, there exists $1 \leq C < 2$ such that for every $s \in \left(0, 1\right]$, the following holds:
    \begin{eqnarray*}
        \left\lVert x' \right\rVert_{\ell_{\phi, \Phi_s}} &=& \left\lVert x_n + x' - x_n \right\rVert_{\ell_{\phi, \Phi_s}}\\
        &\leq& C^{1/s} \left(\left\lVert x_n \right\rVert_{\ell_{\phi, \Phi_s}} + \left\lVert x' - x_n \right\rVert_{\ell_{\phi, \Phi_s}}\right)\\
        &<& C^{1/s} \left(\infty + \infty\right)\\
        &<& \infty.
    \end{eqnarray*}
    This implies $x' \in \ell_{\phi, \Phi_s}$. Furthermore, since $\left\lVert x_n - x' \right\rVert_{\ell_{\phi, \Phi_s}} < \varepsilon$, by the definition of convergence in normed spaces, it follows that $\lim x_k = \lim_{j \to \infty} x_j^{(k)} = x'$. Since $\left(x_k\right)_{k \in \mathbb{Z}} \subset \ell_{\phi, \Phi_s}$ is a Cauchy sequence, $x' \in \ell_{\phi, \Phi_s}$, and $\lim x_k = x'$, it can be concluded that $\ell_{\phi, \Phi_s}$ is a complete space. Therefore, $\left(\ell_{\phi, \Phi_s}, \left\lVert \cdot \right\rVert_{\ell_{\phi, \Phi_s}}\right)$ is a quasi-Banach space.
\end{proof}

\section{Conclusion}
The research findings regarding the discrete Orlicz-Morrey-\emph{s} spaces $\ell_{\phi. \Phi_s}$ and the properties of sequences in these spaces suggest that the discrete Orlicz-Morrey-\emph{s} spaces $\ell_{\phi, \Phi_s}$ generalize the discrete Orlicz-Morrey spaces $\ell_{\phi, \Phi}$. The discrete Orlicz-Morrey-\emph{s} spaces $\ell_{\phi, \Phi_s}$ are equipped with the $\left\lVert \cdot \right\rVert _{\ell_{\phi, \Phi_s}}$ function, which defines a quasi-norm on $\ell_{\phi, \Phi_s}$. Furthermore, the lemmas or properties of the disrete Orlicz-Morrey spaces hold for the discrete Orlicz-Morrey-\emph{s} spaces under certain conditions. Finally, it is concluded that the discrete Orlicz-Morrey-\emph{s} spaces is complete, and consequently, the space $\left(\ell_{\phi, \Phi_s}, \left\lVert \cdot \right\rVert_{\ell_{\phi, \Phi_s}}\right)$ is a Quasi-Banach space.

\section*{Acknowledgments}
This research is supported by Penelitian Pengembangan Kelompok Bidang Keilmuan and Penguatan Kompetensi Universitas Pendidikan Indonesia (UPI) 2024.

\bibliographystyle{plain}
\bibliography{references}

\end{document}